\documentclass[a4paper,12pt]{amsart}
\usepackage{amssymb}
\usepackage{ifthen}
\usepackage{graphicx}
\usepackage{float}
\usepackage{caption}
\usepackage{subcaption}
\usepackage{cite}
\usepackage{amsfonts}
\usepackage{amscd}
\usepackage{amsxtra}
\usepackage{color}
\newtheorem{thm}{Theorem}
\newtheorem{cor}{Corollary}
\newtheorem{lem}{Lemma}
\newtheorem{rem}{Remark}

\newtheorem{conj}{Conjecture}
\newtheorem{prob}{Problem}

\theoremstyle{definition}
\newtheorem{defn}{Definition}[section]
\newtheorem{example}{Example}

\newenvironment{pf}[1][]{%
 \vskip 1mm
 \noindent
 \ifthenelse{\equal{#1}{}}%
  {{\slshape Proof. }}%
  {{\slshape #1.} }%
 }%
{\qed\bigskip}

\newcounter{alphabet}
\newcounter{tmp}
\newenvironment{Thm}[1][]{\refstepcounter{alphabet}%
\bigskip%
\noindent%
{\bf Theorem \Alph{alphabet}}%
\ifthenelse{\equal{#1}{}}{}{ (#1)}%
{\bf .} \itshape}{\vskip 8pt}

\makeatletter
\renewcommand{\Ref}[1]{\@ifundefined{r@#1}{}{\setcounter{tmp}{\ref{#1}}\Alph{tmp}}}
\makeatother

\newenvironment{Lem}[1][]{\refstepcounter{alphabet}%
\bigskip%
\noindent%
{\bf Lemma \Alph{alphabet}}%
{\bf .} \itshape}{\vskip 8pt}

\newcommand{\IC}{{\mathbb C}}
\newcommand{\ID}{{\mathbb D}}

\def\be{\begin{equation}}
\def\ee{\end{equation}}

\newcommand{\bee}{\begin{enumerate}}
\newcommand{\eee}{\end{enumerate}}

\newcommand{\blem}{\begin{lem}}
\newcommand{\elem}{\end{lem}}
\newcommand{\bthm}{\begin{thm}}
\newcommand{\ethm}{\end{thm}}
\newcommand{\bcor}{\begin{cor}}
\newcommand{\ecor}{\end{cor}}
\newcommand{\beg}{\begin{example}}
\newcommand{\eeg}{\end{example}}
\newcommand{\begs}{\begin{examples}}
\newcommand{\eegs}{\end{examples}}
\newcommand{\bdefe}{\begin{defn}}
\newcommand{\edefe}{\end{defn}}
\newcommand{\bprob}{\begin{prob}}
\newcommand{\eprob}{\end{prob}}
\newcommand{\bques}{\begin{ques}}
\newcommand{\eques}{\end{ques}}
\newcommand{\bei}{\begin{itemize}}
\newcommand{\eei}{\end{itemize}}
\newcommand{\bcon}{\begin{conj}}
\newcommand{\econ}{\end{conj}}
\newcommand{\bcons}{\begin{conjs}}
\newcommand{\econs}{\end{conjs}}
\newcommand{\bprop}{\begin{propo}}
\newcommand{\eprop}{\end{propo}}
\newcommand{\br}{\begin{rem}}
\newcommand{\er}{\end{rem}}
\newcommand{\brs}{\begin{rems}}
\newcommand{\ers}{\end{rems}}
\newcommand{\bo}{\begin{obser}}
\newcommand{\eo}{\end{obser}}
\newcommand{\bos}{\begin{obsers}}
\newcommand{\eos}{\end{obsers}}
\newcommand{\bpf}{\begin{pf}}
\newcommand{\epf}{\end{pf}}
\newcommand{\ba}{\begin{array}}
\newcommand{\ea}{\end{array}}
\newcommand{\beq}{\begin{eqnarray}}
\newcommand{\beqq}{\begin{eqnarray*}}
\newcommand{\eeq}{\end{eqnarray}}
\newcommand{\eeqq}{\end{eqnarray*}}

\newcommand{\ra}{\rightarrow}

\def\cc{\setcounter{equation}{0}   
\setcounter{figure}{0}\setcounter{table}{0}}

\newcounter{minutes}
\divide\time by 60
\newcounter{hours}
\multiply\time by 60 \addtocounter{minutes}{-\time}

\begin{document}
\bibliographystyle{amsplain}
\title[Improved Bohr's inequality for mappings defined on simply connected domains]{Improved Bohr's inequality for simply connected domains}

\thanks{
File:~\jobname .tex,
          printed: \number\day-\number\month-\number\year,
          \thehours.\ifnum\theminutes<10{0}\fi\theminutes}




\author[S. Evdoridis]{Stavros Evdoridis}
\address{S. Evdoridis, Aalto University,
Department of Mathematics and Systems Analysis, P. O. Box 11100,
FI-00076 Aalto, Finland.}
\email{stavros.evdoridis@aalto.fi}

\author[S. Ponnusamy]{Saminathan Ponnusamy}
\address{S. Ponnusamy, Department of Mathematics,
Indian Institute of Technology Madras, Chennai-600 036, India}
\email{samy@iitm.ac.in}

\author[A. Rasila]{Antti Rasila${}^{~\mathbf{*}}$}
\address{A. Rasila,  Technion -- Israel Institute of Technology, Guangdong Technion,
241 Daxue Road, Shantou 515063, Guangdong, People's Republic of China} \email{antti.rasila@iki.fi; antti.rasila@gtiit.edu.cn}


\keywords{Bounded analytic functions, harmonic functions, locally univalent functions and Bohr radius.
\\
${}^{\mathbf{*}}$ Corresponding author
}

\begin{abstract}
In this paper, we study the Bohr phenomenon for functions that are defined on a general simply connected domain of the complex plane.
We improve known results of R. Fournier and St. Ruscheweyh for a class of analytic functions. Furthermore, we examine the case where a harmonic
mapping is defined in a disk containing $\mathbb{D}$ and obtain a Bohr type inequality.
\end{abstract}

\thanks{The research was supported by Academy of Finland and NNSF of China (No. 11971124)
}

\maketitle
\pagestyle{myheadings}
\markboth{S. Evdoridis, S. Ponnusamy, and  A. Rasila}{Bohr's inequality for simply connected domains}
\cc

\section{Introduction and Main Results}
Let $\ID (a;r)=\{z:\,|z-a|<r\}$, and let  $\ID :=\ID (0;1)$, the open unit disk in the complex plane $\IC$.
For a given simply connected domain $\Omega$ containing $\ID$, let ${\mathcal H}(\Omega)$ denote the class of
analytic functions on $\Omega$, and let ${\mathcal B}(\Omega)$ be the class of functions $f\in{\mathcal H}(\Omega)$
such that $  f(\Omega ) \subseteq \overline{\ID}$. The Bohr radius for the family ${\mathcal B}(\Omega)$ is defined to be
the positive real number $B=B_\Omega \in (0,1)$ given by (see \cite{FourRush-2010})
$$B=\sup \{r\in (0,1):\, M_f(r)\leq 1 ~\mbox{ for all $f(z)=\sum_{n=0}^{\infty} a_n z^n \in {\mathcal B}(\Omega)$, $z\in\ID$}\},
$$
where $M_f(r)=\sum_{n=0}^{\infty} |a_n|\, r^n$ is the majorant series associated with $f\in {\mathcal B}(\Omega)$ in $\ID$.
If $\Omega =\ID$, then it is well-known that $B_{\ID}=1/3$, and it is
described precisely as follows:

\begin{Thm} \label{EPR2-thA} {\rm (The Classical Bohr (radius 1/3) Theorem)}
If $f\in {\mathcal B}(\ID)$,  then  $M_f(r)\le 1$ for $0\leq r\leq 1/3$. The number $1/3$ is best possible.
\end{Thm}

The inequality $M_f(r)\le 1$, for $f\in {\mathcal B}(\ID)$, fails to hold for any $r>1/3$. This can be seen by considering
the function $\varphi_a (z)=(a-z)/(1-az)$ and by taking $a\in (0,1)$ such that $a$ sufficiently close to $1$.

Theorem \Ref{EPR2-thA} was originally obtained by H. Bohr in 1914 \cite{Bohr-14} for $0\leq r\leq 1/6$.
The optimal value $1/3$, which is called the Bohr radius for the disk case, was
later established independently by M. Riesz, I. Schur, and F.W. Wiener. Proofs have also been given by Sidon \cite{Sidon-27-15}
and  Tomi\'c \cite{Tomic-62-16}.
Over the past two decades there has been significant interest in Bohr-type inequalities. See
\cite{AlKayPON-19, BenDahKha, BS1962, BomBor-04, DFOOS, EvPoRa-19, FourRush-2010,
GarMasRoss-2018, KayPon1, KaPo-18, KayPon2, KP-AASFM2-19, LP2018, LPW2020, LP2019, PVW, PW}
and the references therein. The paper \cite{BoKha1997} carried Bohr's theorem to prominence for the case of several complex variables. A series of papers by several authors followed this article extending and generalizing this phenomenon
to many different situations. Readers are referred to \cite{AAPon1} and \cite[Chapter 8]{GarMasRoss-2018}
for more information about Bohr's inequality and related investigations.

For $0\leq \gamma <1$, we consider the disk $\Omega_{\gamma}$ defined by
$$\Omega_{\gamma}=\left \{z\in \mathbb{C}:\, \left |z+\frac{\gamma}{1-\gamma}\right |<\frac{1}{1-\gamma}\right \}.
$$
It is clear that the unit disk $\ID$ is always a subset of $\Omega_{\gamma}$.
In 2010, Fournier and Ruscheweyh \cite{FourRush-2010} extended Bohr's inequality in the following form.

\begin{Thm}{\rm (\hspace{1sp}\cite[Theorem 1]{FourRush-2010})}\label{FR2010-thB}
For $0\leq \gamma <1$, let $f\in {\mathcal B}(\Omega_\gamma)$ with $f(z)=\sum_{n=0}^\infty a_nz^n$ in $\mathbb{D}$. Then,
$$\sum_{n=0}^\infty |a_n|r^n \leq 1  \text{ for } r\leq \rho_\gamma :=\frac{1+\gamma}{3+\gamma}.
$$
Moreover, $\sum_{n=0}^\infty |a_n|\rho_{\gamma}^n =1$ holds for a function $f(z)=\sum_{n=0}^\infty a_nz^n$ in
${\mathcal B}(\Omega_\gamma)$ if and only if $f(z)=c$ with $|c|=1$.
\end{Thm}

We are now in a position to state an improved version of this result.

\bthm\label{EPR2-th1}
For $0\leq \gamma <1$, let $f\in {\mathcal B}(\Omega_\gamma)$ with $f(z)=\sum_{n=0}^\infty a_nz^n$ in $\mathbb{D}$. Then, we have
$$\sum_{n=0}^\infty |a_n|r^n +  \frac{8}{9}\left(\frac{S_{r(1-\gamma)}}{\pi}\right) \leq 1  \text{ for } r\leq \frac{1+\gamma}{3+\gamma},
$$
where $S_r$ denotes the area of the image of the disk $\ID(0;r)$ under the mapping $f$. Moreover, the inequality is strict unless
$f$ is a constant function. The bound  $8/9$ and the number $(1+\gamma)/(3+\gamma)$ cannot be replaced by a larger quantity.
\ethm

\bcon
We conjecture that the constant $8/9$ can be replaced by a decreasing function $t(\gamma)$ from $[0,1)$ onto $[8/9, 16/9)$.
\econ

For an analytic function $f(z)=\sum_{n=0}^\infty a_n z^n$ in the unit disk, we write
$$\|f_0\|_r = \sum_{n=1}^\infty |a_n|^2 r^{2n},
$$
where $f_0(z)=f(z)-f(0)$. Very recently, the authors have shown in  \cite{PVW} an improvement of Theorem \Ref{EPR2-thA}:
if $f\in{\mathcal B}(\ID)$, then for every $r\leq 1/3$
$$\sum_{n=0}^\infty |a_n|r^n +\left( \frac{1}{1+|a_0|}+ \frac{r}{1-r}\right)\|f_0\|_r \leq 1.
$$
In the next result we extend this improved version to $ {\mathcal B}(\Omega_\gamma)$
and thus, we have the following refinement of Theorem \Ref{FR2010-thB}.

\begin{thm}\label{EPR2-th2}
For $0\leq \gamma <1$, let $f\in {\mathcal B}(\Omega_\gamma)$ with $f(z)=\sum_{n=0}^\infty a_nz^n$ in $\mathbb{D}$. Then, we have
$$\sum_{n=0}^\infty |a_n|r^n +\left( \frac{1}{1+|a_0|}+ \frac{r}{1-r}\right)\|f_0\|_r \leq 1 \text{  for  } r\leq r_0:= \frac{1+\gamma}{3+\gamma},$$
and the number $r_0$ cannot be improved.
\end{thm}

The next result concerns a more general case, where the function under consideration is analytic in a simply connected domain
$\Omega$, containing the unit disk $\ID$. Then, as in \cite{FourRush-2010}, we introduce
\be\label{EPR2-eq3}
\lambda = \lambda (\Omega) = \sup_{\substack{{f\in \mathcal{B}(\Omega)} \\ {n\geq 1}}}
\left \{\frac{|a_n|}{1-|a_0|^2} :\, a_0 \not\equiv f(z)=\sum_{n=0}^\infty a_nz^n, ~ z\in\mathbb{D}\right\}.
\ee

\bthm\label{EPR2-th3}
Let $\Omega\supset\mathbb{D}$ be a simply connected domain and $f\in \mathcal{B}(\Omega)$, with
$f(z)=\sum_{n=0}^\infty a_nz^n$ for $z\in \mathbb{D}$. Then, we have
\be\label{EPR2-eq4}
B_1(r):=\sum_{n=0}^\infty |a_n|r^n + 2\left( \frac{1+\lambda}{1+2\lambda}\right)^2 \, \frac{S_r}{\pi} \leq 1
~\mbox{ for }~ r\leq \frac{1}{1+2\lambda},
\ee
where $S_r$ denotes the area of the image of the disk $\mathbb{D}(0,r)$ under the mapping $f$.
\ethm

Recall that  a complex-valued function $f=u+iv$ in a simply connected domain $\Omega$ is called
harmonic in $\Omega$ if it satisfies the Laplace equation $\triangle f=4f_{z\overline{z}}=0$, i.e., $u$ and $v$ are real harmonic
in $\Omega$.
It follows that every harmonic mapping $f$ admits a representation of the form $f=h+\overline{g}$, where $h$ and $g$ are analytic in $\Omega$.
This representation is unique up to an additive constant.  The Jacobian $J_{f}$ of $f$ is given by $J_{f}(z)=|h'(z)|^{2}-|g'(z)|^{2}$.

We say that a locally univalent function $f$ is sense-preserving if $J_{f}(z)>0$ in $\Omega$. Consequently,
a harmonic mapping $f$ is locally univalent and sense-preserving in $\Omega$ if and only if $J_{f}(z)>0$ in $\Omega$; or
equivalently if $h'\neq 0$ in $\Omega$ and the dilatation $\omega_{f}:=g'/h'$ of $f$ has the property that $|\omega_{f}|<1$ in $\Omega$ \cite{L1936}.


If a locally univalent and sense-preserving harmonic mapping $f=h+\overline{g}$ on $\Omega$ satisfies the condition
$ |\omega_{f}(z)| \leq k<1$ for $\Omega$,
then $f$ is called $K$-quasiregular harmonic mapping on $\Omega$, where $K=(1+k)/(1-k)\geq 1$
(cf. \cite{Kalaj2008,Martio1968}).
Obviously $k\rightarrow 1$ corresponds to the limiting case $K\rightarrow \infty$.
Harmonic extensions of the classical Bohr theorem have been established in \cite{EvPoRa-19,KayPon2,KaPoSha,LPW2020,LP2019}.

In the following, we consider a harmonic mapping in $\Omega_\gamma$ and obtain Bohr's inequality for its restriction to the unit disk.

\begin{thm}\label{EPR2-th4}
Let $f=h+\overline{g}$ be a harmonic mapping in $\Omega_\gamma$, with $|h(z)|\leq 1$ on $\Omega_\gamma$.
If $h(z)=\sum_{n=0}^\infty a_nz^n$ and $g(z)=\sum_{n=1}^\infty b_nz^n$ in $\mathbb{D}$ and  $|g'(z)|\leq k|h'(z)|$ for some
$k\in [0,1]$, then
$$\sum_{n=0}^\infty |a_n|r^n + \sum_{n=1}^\infty |b_n|r^n \leq 1 \text{  for  } r\leq r_0:=\frac{1+\gamma}{3 +2k+\gamma}.
$$
The radius $r_0$ is the best possible.
\end{thm}

\begin{cor}
Let $f=h+\overline{g}$ be a harmonic mapping in $\Omega_\gamma$, with $|h(z)|\leq 1$ on $\Omega_\gamma$.
If $h(z)=\sum_{n=0}^\infty a_nz^n$,  $g(z)=\sum_{n=1}^\infty b_nz^n$ in $\mathbb{D}$ and $f=h+\overline{g}$
is sense-preserving in $\ID$, then
$$\sum_{n=0}^\infty |a_n|r^n + \sum_{n=1}^\infty |b_n|r^n \leq 1 \text{  for  } r\leq r_0:=\frac{1+\gamma}{5+\gamma}.
$$
The radius $r_0$ is the best possible.
\end{cor}

\section{Proofs of the main results}

\subsection{Necessary Lemma}
Pick's conformally invariant form of Schwarz's lemma states that if  $f\in {\mathcal B}(\ID)$, then
\begin{equation}\label{AKP-SchPick}
|f(z)|\le \frac{r+|f(0)|}{1+|f(0)| r} ~\mbox{ and }~
|f'(z)|\le \frac{1-|f(z)|^2}{1-|z|^2},\quad z\in\ID.
\end{equation}
Furthermore, it is well-known that the Taylor coefficients of  $f\in {\mathcal B}(\ID)$ satisfy the inequalities:
\begin{equation}\label{AKP-SchPick2}
\frac{|f^{(n)}(0)|}{n!}\le 1-|f(0)|^2 ~\mbox{ for any $n\ge 1$.}
\end{equation}
More generally, we have the following sharp estimate for higher order derivatives.

\begin{Lem}
{\rm (Ruscheweyh \cite{Ru})}
\label{R-SchPick1}
For  $f\in {\mathcal B}(\ID)$, we have
$$\frac{|f^{(n)}(\alpha)|}{n!}\le \frac{1-|f(\alpha)|^2}{(1-|\alpha|)^{n-1}(1-|\alpha|^2)}
$$
for each $n\ge 1$ and $\alpha\in\ID$.
Moreover, for each fixed $n\ge 1$ and $\alpha\in \ID$,
$$\sup_{f} (1-|\alpha|)^{n-1}\frac{|f^{(n)}(\alpha)|}{n!} \frac{1-|\alpha|^2}{1-|f(\alpha)|^2} =1,
$$
where the supremum is taken over all nonconstant analytic functions $f\in {\mathcal B}(\ID)$.
\end{Lem}

\begin{lem}\label{EPR2-lem1}
Let $g:\,\mathbb{D}\to \mathbb{\overline{D}}$ be an analytic function, and let $\gamma \in \mathbb{D}$ be such that
$g(z)=\sum_{n=0}^\infty \alpha_n (z-\gamma )^n$ for $|z-\gamma | <1-|\gamma |$. Then
$$\sum_{n=0}^\infty  |\alpha_n|r^n +\frac{8}{9}\left(\frac{S_{r}^\gamma}{\pi} \right) \leq 1  ~\mbox{  for  }~ r\leq r_0:=\frac{1-|\gamma |^2}{3+|\gamma |} ,
$$
where $S_{r}^\gamma$ denotes the area of the image of the disk $\ID(\gamma; r(1-|\gamma |))$ under the mapping $g$.
\end{lem}
\begin{proof}
Without loss of generality, we may assume that $\gamma \in [0,1)$. Also, we note that $z\in D_\gamma:=\ID (\gamma; 1-\gamma )$
if and only if $w=(z-\gamma)/(1-\gamma)$ lies inside $\mathbb{D}$.

For $z\in D_\gamma $, define $\phi :\,D_{\gamma}\to \mathbb{D}$ by $w=\phi (z)=\frac{z-\gamma}{1-\gamma}$.
Then we have
\begin{eqnarray*}
g(z)
&=&\sum_{n=0}^\infty \alpha_n (1-\gamma)^n \phi (z)^n =\sum_{n=0}^\infty b_n \phi (z)^n =:G(\phi (z)),
\end{eqnarray*}
in $z\in D_{\gamma}$, where $G(w)$ is an analytic function in $\mathbb{D}$ with
$$G(w)=\sum_{n=0}^\infty b_nw^n ~\mbox{ for $w\in\mathbb{D}$,}
$$
so that $g=G\circ \phi $ in $D_{\gamma}$.\\
As in \cite{KaPo-18}, for the area term of the function $G$, we have the upper bound
\be\label{EPR2-eq1}
\frac{S_{r}^\gamma }{\pi} = \frac{1}{\pi}{\rm Area}\left[G \big( \ID(0;r)\big) \right] \leq (1-|b_0|^2)^2 \frac{r^2}{(1-r^2)^2}=(1-|\alpha_0|^2)^2 \frac{r^2}{(1-r^2)^2}.
\ee
Moreover,
\begin{eqnarray*}
{\rm Area}\left[ G\big( \ID(0;r)\big)\right] &=& {\rm Area} \left[g\big(\phi ^{-1}( \ID(0;r))\big)\right]
= {\rm Area} \left[g\big( \ID(\gamma;r(1-\gamma))\big)\right].
\end{eqnarray*}
By Lemma \Ref{R-SchPick1}, it follows that
$$ |\alpha_n|=\frac{|g^{(n)}(\gamma)|}{n!}\leq (1-|\alpha_0|^2)\frac{(1+\gamma)^{n-1}}{(1-\gamma^2)^n}= \frac{1-|\alpha_0|^2}{(1+\gamma)(1-\gamma)^n}
$$
for the Taylor coefficients of $g$, and therefore, we have the inequality
\be\label{EPR2-eq2}
\sum_{n=1}^\infty |\alpha_n|r^n \leq \frac{1-|\alpha_0|^2}{1+\gamma}\sum_{n=1}^\infty \frac{r^n}{(1-\gamma)^n}
=\frac{1-|\alpha_0|^2}{1+\gamma}\left ( \frac{r}{1-\gamma -r}\right )
.
\ee
Thus, by \eqref{EPR2-eq1} and \eqref{EPR2-eq2}, we find that
$$\sum_{n=0}^\infty |\alpha_n|r^n +K\left(\frac{S_r^{\gamma}}{\pi}\right) \leq 1+ \Psi (r)
$$
which is less than or equal to $1$ provided that $\Psi (r)\leq 0$, where
$$\Psi (r)=|\alpha_0|+\frac{(1-|\alpha_0|^2)r}{(1+\gamma)(1-\gamma -r)}+K\frac{(1-|\alpha_0|^2)^2r^2}{(1-r^2)^2} -1,
$$
and $|\alpha_0|\leq 1$, which is always true by the assumption. Now, it is a simple exercise to see that $\Psi$ is an increasing function of $r$
for all $r<1-\gamma$, which in turn implies that
$$\Psi (r)\leq \Psi (r_0) \mbox{  for  }~ r\leq r_0:=\frac{1-\gamma ^2}{3+\gamma } .
$$
Thus, to prove that $\Psi (r)\leq 0$, it suffices to show that $\Psi (r_0)\leq 0$ for all $|\alpha_0|\leq 1$. Elementary computation gives that
\beqq
\Psi (r_0) &= & \frac{1-|\alpha_0|^2}{2}\left [1 + 2K(1-|\alpha_0|^2)\frac{(3+\gamma
)^2(1-\gamma ^2)^2}{[(3+\gamma)^2-(1-\gamma ^2)^2]^2}-\frac{2}{1+|\alpha_0|}	 \right ]\\
&=&\frac{1-|\alpha_0|^2}{2}F(|\alpha_0|),
\eeqq
where
$$F(x)= 1+ 2KA^2(\gamma)(1-x^2)-\frac{2}{1+x}, \quad x\in [0,1],
$$
and
$$A(\gamma)=\frac{(3+\gamma )(1-\gamma ^2)}{(3+\gamma)^2-(1-\gamma ^2)^2}.
$$

It remains to show that $F(x) \leq 0$ for all $x\in [0,1]$ and $\gamma \in [0,1)$.
To do this, we first observe that $A(\gamma)>0$ for $\gamma \in [0,1)$,
$$ F(0)=2KA^2(\gamma)-1 ~\mbox{ and }~ \lim_{x\to 1^-}F(x)=0.
$$
Also, we may write
$$A(\gamma) =(M\circ N)(\gamma ), \quad M(r)=\frac{r}{1-r^2}, ~\mbox{ and }~ N(\gamma )=\frac{1-\gamma ^2}{3+\gamma }.
$$
It follows that $A'(\gamma) =M'(N(\gamma )) N'(\gamma )$, where
$$N'(\gamma ) = -\left (\frac{\gamma ^2+6 \gamma +1}{(3+\gamma )^2}\right ),
$$
showing that $M(r)$ is an increasing function of $r$ in $(0,1)$ and $N$ is a decreasing function of $\gamma$ in $[0,1)$. Therefore, it
follows that $A$ is a decreasing function of $\gamma$ in $[0,1)$, with $A(0)=3/8$ and $A(1)=0$,
and thus, $A^2(\gamma)$ is decreasing on $[0,1]$. Hence, we have
$$A^2(\gamma)\leq A^2(0)=\frac{9}{64}.
$$
Finally, we obtain that
\beqq
F'(x)&=&-4KA^2(\gamma)x+\frac{2}{(1+x)^2}\\
&=&\frac{2}{(1+x)^2}[1-2KA^2(\gamma)x(1+x)^2]\\
&\geq& \frac{2}{(1+x)^2}\left (1-K\frac{9}{8}\right ),
\eeqq
which  is positive for all $x\in (0,1)$
whenever $K\leq 8/9$.
Thus, $F$ is increasing on $[0,1]$ for $K\leq 8/9$. In particular, $F(x) \leq 0$ for all $x\in [0,1]$ and $\gamma \in [0,1)$.
This completes the proof of the lemma.
\end{proof}

\subsection{Proof of Theorem \ref{EPR2-th1}}
For $0\leq \gamma<1$, let
\[
\Omega_{\gamma}=\left \{z\in \mathbb{C}:\, \left |z+\frac{\gamma}{1-\gamma}\right |<\frac{1}{1-\gamma}\right \}
\]
and consider $f:\, \Omega_{\gamma}\to \mathbb{\overline{D}}$ as in the statement.
We consider the analytic function $\phi :\,\mathbb{D}\to \Omega_{\gamma}$ defined by  $\phi(z)=(z-\gamma)/(1-\gamma)$. Then the
composition $g=f\circ \phi$ is analytic in $\mathbb{D}$ and
$$g(z)=(f\circ \phi)(z)= \sum_{n=0}^\infty
 \frac{a_n}{(1-\gamma)^n} \, (z-\gamma)^n ~\mbox{ for $|z-\gamma|<1-\gamma$}.
$$
Applying Lemma \ref{EPR2-lem1} to the function $g$ gives
$$\sum_{n=0}^\infty \left| \frac{a_n}{(1-\gamma)^n}\right | \rho ^n + \frac{8}{9}\left( \frac{S_\rho^{\gamma}}{\pi}\right)
\leq 1 ~\mbox{ for }~ \rho \leq \frac{1-\gamma ^2}{3+\gamma},
$$
where $S_\rho^{\gamma}$ is defined as in Lemma \ref{EPR2-lem1} for $g$; or equivalently,
$$ \sum_{n=0}^\infty | a_n| \left( \frac{\rho}{1-\gamma}\right)^n + \frac{8}{9}\left( \frac{S_\rho^{\gamma}}{\pi}\right)
\leq 1~\mbox{ for }~ \frac{\rho}{1-\gamma}\leq \frac{1+\gamma }{3+\gamma}.
$$
The desired inequality follows by setting $\rho =r(1-\gamma)$.

In order to prove the sharpness of the result, we consider the composition of the functions $G:\,\Omega_\gamma \to \mathbb{D}$ with $G(z)=(1-\gamma)z+\gamma$ and $\psi :\mathbb{D} \to \mathbb{D}$ with
\[
\psi(z)=\frac{a-z}{1-az},
\]
for $a\in (0,1)$.
Then $g_0= \psi \circ G$ maps $\Omega_\gamma$ univalently onto $\mathbb{D}$. This gives
$$g_0(z)=\frac{a-\gamma-(1-\gamma)z}{1-a\gamma-a(1-\gamma)z}=A_0-\sum_{n=1}^\infty A_n z^{n}, ~z\in\ID,
$$
where $a\in (0,1)$,
\be\label{EPR2-eq6}
A_0=\frac{a-\gamma}{1-a\gamma} ~\mbox{ and }~A_n=\frac{1-a^2}{a(1-a\gamma)}\left( \frac{a(1-\gamma)}{1-a\gamma}\right)^n
\ee

For a given $\gamma \in [0,1)$, let $a>\gamma$ and find that
\begin{eqnarray*}
\sum_{n=0}^\infty|a_n|r^n +\frac{8}{9}\left( \frac{S_{r(1-\gamma)}}{\pi}\right) &=&
A_0+\sum_{n=1}^\infty A_n r^{n} + \frac{8}{9} \sum_{n=1}^\infty nA_n^2 r^{2n}\\
&=& \frac{a-\gamma}{1-a\gamma}+\frac{1-a^2}{1-a\gamma}\frac{(1-\gamma)r}{1-a\gamma-ar(1-\gamma)}\\
& &+ \frac{8}{9}\frac{r^2(1-a^2)^2(1-\gamma)^4}{[(1-a\gamma)^2-a^2r^2(1-\gamma)^4]^2}\\
&=&1-(1-a)\Phi (r),
\end{eqnarray*}
where
$$\Phi (r)=\frac{1+\gamma}{1-a\gamma}-\frac{1+a}{1-a\gamma}\frac{r(1-\gamma)}{1-a\gamma -ar(1-\gamma)}
-\frac{8}{9}\frac{(1-a)(1+a)^2(1-\gamma)^4r^2}{[(1-a\gamma)^2-a^2r^2(1-\gamma)^4]^2}.
$$
Then, $\Phi$ is a strictly decreasing function of $r$ in $(0,1)$ and thus, for $r>r_0=(1+\gamma)(3+\gamma)$, we have
\begin{eqnarray*}
\Phi (r)<\Phi(r_0) &=& \frac{1+\gamma}{1-a\gamma}-\frac{1+a}{1-a\gamma}\frac{1-\gamma ^2}{(1-a\gamma)(3+\gamma)-a(1-\gamma ^2)}\\
& & -\frac{8}{9}(1-a)\frac{(1+a)^2(1-\gamma)^4(1+\gamma)^2(3+\gamma)^2}{[(1-a\gamma)^2(3+\gamma)^2-a^2(1-\gamma)^4(1+\gamma)^2]^2},
\end{eqnarray*}
which tends to 0 as $a\to 1$. Therefore, $\Phi $ is negative for $r>r_0$ and hence, $1-(1-a)\Phi(r)>1$.
\hfill $\Box$

\vspace {6pt}

For the proof of Theorem \ref{EPR2-th2}, we modify the previous arguments slightly and prove the following:

\begin{lem} \label{lem*}
For $\gamma\in [0,1) $, let
$$\Omega_{\gamma}:=\left \{z\in \mathbb{C}:| z+ \frac{\gamma}{1-\gamma}|<\frac{1}{1-\gamma}\right \},
$$
and let $f$ be an analytic function in $\Omega_\gamma$, bounded by $1$, with the series representation
$f(z)=\sum_{n=0}^\infty a_n z^n$ in the unit disk $\ID$. Then,
$$|a_n|\leq \frac{1-|a_0|^2}{1+\gamma} ~\mbox{ for $n\geq 1$.}
$$
\end{lem}
\begin{proof}
Clearly, $\phi(z)=(z-\gamma)/(1-\gamma)\in \Omega_{\gamma}$ if and only if $z\in\ID$. Then $g:\,\ID\ra \ID$ defined by
$g(z)=(f\circ \phi)(z)$,
is analytic in  $\ID$, $g(\gamma)=f(0)=a_0$ and
$$\frac{g^{(n)}(z)}{n!}=f^{(n)}\left(\frac{z-\gamma}{1-\gamma}\right) \frac{1}{(1-\gamma)^n}.
$$
In particular, at $z=\gamma$, this gives
$$(1-\gamma)^n\frac{g^{(n)}(\gamma)}{(n!)^2}=\frac{f^{(n)}(0)}{n!}=a_n   ~\mbox{ for $n\geq 1$.}~
$$
%
By Lemma \Ref{R-SchPick1} we deduce that for $n\geq 1$,
$$|a_n| =(1-\gamma)^n\frac{|g^{(n)}(\gamma)|}{n!} \leq \frac{1-|g(\gamma)|^2}{(1+\gamma)}=\frac{1-|a_0|^2}{1+\gamma},
$$
and this completes the proof.
\end{proof}

\subsection{Proof of Theorem \ref{EPR2-th2}}
Without loss of generality, we may assume that $a_0:=a \in (0,1)$. By applying Lemma \ref{lem*} we obtain
\begin{eqnarray*}
M_f(r)&=&\sum_{n=0}^\infty |a_n|r^n +\left( \frac{1}{1+|a_0|}+ \frac{r}{1-r}\right)\|f_0\|_r  \\
&\leq & a +  \frac{1-a^2}{1+\gamma} \sum_{n=1}^\infty r^n
 + \left(\frac{1}{1+a}+ \frac{r}{1-r}\right) \left( \frac{1-a^2}{1+\gamma}\right)^2 \sum_{n=1}^\infty   r^{2n} \\
 &=&a + \frac{1-a^2}{1+\gamma} \frac{r}{1-r} + \left(\frac{1}{1+a}+ \frac{r}{1-r}\right) \left( \frac{1-a^2}{1+\gamma}\right)^2 \frac{r^2}{1-r^2} =u(a) ~\mbox{ (say)}.\\
\end{eqnarray*}
Here $u(a)= a +A(1-a^2) +B(1-a)(1-a^2) +C(1-a^2)^2$ for $a\in [0,1]$,  where
$$A =\frac{1}{1+\gamma}\frac{r}{1-r}, ~B=\frac{1}{(1+\gamma)^2}\frac{r^2}{1-r^2}~\mbox{ and }~C=\frac{1}{(1+\gamma)^2}\frac{r^3}{(1-r)(1-r^2)}.
$$
Now,
$$u'(a)=1-2Aa+ B(3a^2-2a-1) +4C(a^3-a)
$$
and
$$u''(a)= -2A+2B(3a-1)+4C(3a^2-1).
$$
Because $B$ and $C$ are non-negative, $u''$ is an increasing function of $a$ in $[0,1]$, it follows that
$$u''(a)\leq u''(1)=-2A +4B +8C=\frac{2r}{(1+\gamma)^2(1-r)(1-r^2)}\Psi(r),
$$
where
$$ \Psi(r)= 4r^2+2r(1-r)-(1+\gamma)(1-r^2)=(1+r)(r(3+\gamma)-(1+\gamma))
$$
which is non-positive for $r\leq r_0=(1+\gamma)/(3+\gamma)$. Thus, we obtain that $u''(a)\leq 0$ for $a\in [0,1]$
and thus, $u'(a)$ is decreasing on $[0,1]$. Therefore, for $r\leq \frac{1+\gamma}{3+\gamma}$, we have
$$u'(a)\geq u'(1)=1-2A
= \frac{1+\gamma-r(3+\gamma)}{(1+\gamma)(1-r)} \geq 0 ~\mbox{ for all $a\in [0,1]$}
$$
from which it follows that $u(a) \leq u(1)=1$ and this proves the stated inequality.

For the sharpness of the radius,  as in the case of Theorem \ref{EPR2-th1}, we consider the function
$$g_0(z)=\frac{a-\gamma-(1-\gamma)z}{1-a\gamma-a(1-\gamma)z}=A_0-\sum_{n=1}^\infty A_n z^{n}, ~z\in\ID,
$$
where $a\in (0,1)$, and $A_n$ ($n\geq 0$) are given by \eqref{EPR2-eq6}. Now, for a given $\gamma \in [0,1)$, we also let
$a>\gamma$.  Then
\begin{eqnarray*}
M_{g_0}(r) &=& A_0+ \sum_{n=0}^\infty A_nr^n +\left( \frac{1}{1+A_0}+ \frac{r}{1-r}\right) \sum_{n=1}^\infty  A_n^2 r^{2n}  \\
&= &  \frac{a-\gamma}{1-a\gamma}+ \frac{1-a^2}{a(1-a\gamma)} \sum_{n=1}^\infty \left( \frac{a(1-\gamma)}{1-a\gamma}\right)^nr^n\\
 &&+\left(\frac{1-a\gamma}{(1+a)(1-\gamma)}+ \frac{r}{1-r}\right)\left( \frac{1-a^2}{a(1-a\gamma)}\right)^2 \sum_{n=1}^\infty \left( \frac{a(1-\gamma)}{1-a\gamma}\right)^{2n}r^{2n} \\
&=& 1- \frac{1-a}{1-a\gamma} \Phi(r),
\end{eqnarray*}
where
\begin{eqnarray*}
\Phi(r)&=& 1+\gamma - \frac{(1+a)(1-\gamma)r}{1-a\gamma -a(1-\gamma)r} \\
&&-\left(\frac{1-a\gamma}{(1+a)(1-\gamma)}+ \frac{r}{1-r} \right)  \frac{(1+a)(1-a^2)}{1-a\gamma} \frac{(1-\gamma )^2r^2}{(1-a\gamma)^2-a^2(1-\gamma)^2r^2}.
\end{eqnarray*}
The function $\Phi$ is strictly decreasing for $r$ in $(0,1)$ and thus, if $r>r_0$, then $\Phi(r)<\Phi(r_0) \to 0$ as $a\to 1$.
Therefore, $\Phi(r)$ is negative for $r>r_0$, as $a$ tends to 1, and hence $M_{g_0}(r)>1$ for $r>r_0$.
\hfill $\Box$

\vspace {6pt}

\subsection{Proof of Theorem \ref{EPR2-th3}}
By the definition of $\lambda$ given by \eqref{EPR2-eq3}, we have
\be\label{EPR2-eq5}
|a_n|\leq \lambda (1-|a_0|^2) ~\mbox{ for $n\geq 1$.}
\ee
For the term $S_r$ we have
\begin{eqnarray*}
\frac{S_r}{\pi} &=& \frac{1}{\pi} \iint_{|z|<r} |f'(z)|^2\, dx\,dy \le \sum_{n=1}^\infty n|a_n|^2r^{2n} \\
&\leq &\lambda ^2 (1-|a_0|^2)^2 \sum_{n=1}^\infty nr^{2n} =\lambda ^2 (1-|a_0|^2)^2 \frac{r^2}{(1-r^2)^2}.
\end{eqnarray*}
Therefore, by using the expression for $B_1(r)$ given by \eqref{EPR2-eq4} and the inequality \eqref{EPR2-eq5}, we obtain
for $r\leq 1/(1+2\lambda)$ that
\begin{eqnarray*}
B_1(r)
&\leq & |a_0| + \lambda (1-|a_0|^2)\frac{r}{1-r}
+ 2 \left( \frac{1+\lambda}{1+2\lambda}\right)^2 \lambda ^2 (1-|a_0|^2)^2\frac{r^2}{(1-r^2)^2} \\
&\leq & |a_0| + \lambda (1-|a_0|^2)\frac{1/(1+2\lambda)}{1-1/(1+2\lambda)}
\\
& &
+ 2 \left( \frac{1+\lambda}{1+2\lambda}\right)^2 \lambda ^2 (1-|a_0|^2)^2\frac{(1/(1+2\lambda)^2}{(1-(1/1+2\lambda)^2)^2} \\
&= & 1- (1-|a_0|) \left[1-\frac{1+|a_0|}{2}-  \frac{(1+|a_0|)(1-|a_0|^2)}{8} \right]\\
&= & 1-\frac{(1-|a_0|^2)}{8}F(|a_0|),
\end{eqnarray*}
where
$$F(x)
= \frac{8}{1+x}-5+x^2.
$$
We see that $F(0)=3$, $F(1)=0$,  $F'(x)\leq 0$ in $[0,1]$ and thus, $F(x)\geq F(1)=0$ for $x\in[0,1]$. This observation
shows that $B_1(r)\leq 1$ for $r\leq 1/(1+2\lambda)$ and the proof of the theorem is completed.
\hfill $\Box$

\vspace {6pt}

\subsection{Proof of Theorem \ref{EPR2-th4}}
The function $h$ is analytic in $\Omega_\gamma$, with $|h(z)|\leq 1$, and hence Lemma \ref{lem*} implies that
$$|a_n|\leq \frac{1-|a_0|^2}{1+\gamma}~\mbox{ for $n\geq 1$.}
$$
By \cite[Lemma 1]{KaPoSha}, the condition  $|g'(z)|\leq k|h'(z)|$ gives that
$$\sum_{n=0}^\infty |b_n|^2r^n \leq k^2\sum_{n=0}^\infty |a_n|^2r^n \text{  for  } r<1,
$$
and thus
$$\sum_{n=1}^\infty |b_n|r^n \leq \sqrt{\sum_{n=0}^\infty |b_n|^2r^n} \sqrt{\sum_{n=1}^\infty r^n}
\leq  \frac{k(1-|a_0|^2)}{1+\gamma} \frac{r}{1-r}.
$$
Therefore, with $|a_0|=a\geq 0$,
\begin{eqnarray*}
N_f(r)&=&\sum_{n=0}^\infty |a_n|r^n + \sum_{n=1}^\infty |b_n|r^n \\
&\leq & a + (1+k)\frac{1-a^2}{1+\gamma}\frac{r}{1-r}\\
&=&1-\frac{1-a}{(1+\gamma)(1-r)}\left[ 1+\gamma- r\{1+\gamma +(1+a)(1+k)\}\right] ,
\end{eqnarray*}
which is less than or equal to $1$ whenever $r\leq r_0(a)$,
where
$$r_0(a)=\frac{1+\gamma}{1+\gamma +(1+a)(1+k)}.
$$
This gives the condition
\[ r\leq \frac{1+\gamma}{3 +2k+\gamma}=r_0(1).
\]

For the sharpness, we consider the function $f_0=h_0+\overline{g_0}$, in $\Omega_\gamma$, where
$$h_0(z)=\frac{a-\gamma-(1-\gamma)z}{1-a\gamma -a(1-\gamma)z} =A_0-\sum_{n=1}^\infty A_n z^{n}, ~z\in\ID,
$$
where $a\in (0,1)$, and $A_n$ ($n\geq 0$) are given by \eqref{EPR2-eq6}, and
$$f_0(z)=k\lambda [h_0(z)-A_0].
$$
%
Thus, we find that
\begin{eqnarray*}
N_{f_0}(r)&=&A_0+(1+k\lambda)\sum_{n=1}^\infty A_n r^{n}\\
&=&\frac{a-\gamma}{1-a\gamma}+ (1+k\lambda ) \frac{1-a^2}{a(1-a\gamma)} \sum_{n=1}^\infty \left( \frac{a(1-\gamma)}{1-a\gamma}\right)^n r^n \\
&=&1-\frac{1-a}{1-a\gamma}\Phi (r),
\end{eqnarray*}
where
$$\Phi (r)=  1+\gamma -(1+\lambda)\frac{(1+a)(1-\gamma )r}{1-a\gamma -a(1-\gamma)r},
$$
which is strictly decreasing for $r\in [0,1)$, and hence, for $r >  r_0(1)$
$$\Phi (r)< \Phi (r_0(1)) = 1+\gamma -(1+k \lambda)\frac{(1+a)(1-\gamma ^2)}{(1-a\gamma )(3 +2k+\gamma)-a(1-\gamma ^2)},$$
which approaches $0$ as $a$ and $\lambda$ tend to $1$.
\hfill $\Box$

\end{document}